\documentclass[11pt,reqno]{amsart}

\usepackage[T1]{fontenc}
\usepackage{lmodern}
\usepackage[a4paper,margin=29mm]{geometry}

\usepackage{amsmath,amssymb,amsthm}
\usepackage{mathrsfs}
\usepackage{booktabs}
\usepackage{microtype}

\usepackage[
colorlinks=true,
linkcolor=blue,
citecolor=blue,
urlcolor=blue
]{hyperref}

\newcommand{\Q}{\mathbb{Q}}
\newcommand{\Z}{\mathbb{Z}}
\newcommand{\R}{\mathbb{R}}
\newcommand{\C}{\mathbb{C}}
\newcommand{\E}{\mathbb{E}}

\newcommand{\OO}{\mathcal{O}}

\DeclareMathOperator{\rd}{rd}
\DeclareMathOperator{\vol}{vol}
\DeclareMathOperator{\covol}{covol}
\DeclareMathOperator{\NN}{N}
\DeclareMathOperator{\Gal}{Gal}
\DeclareMathOperator{\Cl}{Cl}
\DeclareMathOperator{\logmap}{Log}

\newtheorem{theorem}{Theorem}[section]
\newtheorem{proposition}[theorem]{Proposition}
\newtheorem{lemma}[theorem]{Lemma}
\newtheorem{corollary}[theorem]{Corollary}

\theoremstyle{remark}
\newtheorem{remark}[theorem]{Remark}

\numberwithin{equation}{section}

\title[Ideal-smooth sets and the sum-product exponent]
{Ideal-smooth sets and an explicit upper bound for the real sum-product exponent}

\author{Kaiqiang Zhang}

\author{Yuanguo Zeng}

\author{Yuyu Wang}

\date{}



\begin{document}
\begin{abstract}
We construct arbitrarily large finite sets $A$ of real algebraic integers
such that both $|A+A|$ and $|AA|$ are at most $|A|^{1.95835}$. The
construction combines truncated ideal-smooth $S$-unit fibres with a
coprime additive factor. A tensor-product rank argument, using a
two-dimensional local feature at each selected prime ideal, controls
the loss in the additive factor, while a direct estimate for an outer
parallel body improves the sumset packing bound. The arithmetic
input is an unramified pro-$2$ tower over a known degree-ten field, with
simultaneous Frobenius cuts controlling the small prime ideals. We use
unconditional Tsfasman--Vl\u adu\c t inequalities for the joint
class-number--regulator cost. All finite numerical comparisons entering
the exponent are certified by outward rational interval arithmetic.
No unproved hypothesis is used.
\end{abstract}
\maketitle

\section{Introduction}
For a finite set $A$ in a field, write
\[
 A+A=\{a+b:a,b\in A\},\qquad AA=\{ab:a,b\in A\}.
\]
The sum-product problem asks how small these two sets can be
simultaneously. Bloom, Sawin, Schildkraut, and Zhelezov constructed
arbitrarily large real sets with a fixed power saving from the
quadratic exponent \cite{BSSZ}. Their construction uses units of
number fields together with an additive lattice-point factor. We
retain this useful interaction, and enlarge the multiplicative
factor to include many truncated prime-ideal valuations.

\begin{theorem}\label{thm:main}
There are finite sets $A_j\subset\R$ of real algebraic integers, with
$|A_j|\to\infty$, such that
\[
 \max\{|A_j+A_j|,|A_jA_j|\}\le |A_j|^{1.95835}.
\]
The construction and the bound are unconditional.
\end{theorem}

In particular, defining
\[
 L_{\mathbb F}=\liminf_{N\to\infty}
 \inf_{\substack{A\subset\mathbb F\\ |A|=N}}
 \frac{\log\max\{|A+A|,|AA|\}}{\log N},
 \qquad \mathbb F\in\{\R,\C\},
\]
we obtain $L_{\R},L_{\C}\le1.95835$. The degrees of the number
fields grow with $j$. The assertion is an existence statement along
a sequence of cardinalities tending to infinity.

The multiplicative factor is a union of unit fibres over truncated
smooth ideals. It must be defined through the $S$-unit valuation
lattice: the relevant prime ideals need not be principal, and
independently chosen generators need not have compatible
archimedean sizes. A single class-group average and a common
logarithmic translate handle both issues. The final set has the
form $A=HP$, where $H$ is this ideal-smooth factor and $P$ is an
additive factor coprime to all selected prime ideals.

There are two geometric improvements. First, a finite-field matrix
argument gives the coprimality sieve a loss of at most $2$ for each
selected prime ideal, without assuming equidistribution of lattice
points in residue classes. Second, rather than enlarging an
$\ell^\tau$ body by pointwise subadditivity, we estimate its actual
outer parallel body. A concave radial minimization reduces the
estimate to strata with at most one exceptional coordinate.
The resulting volume loss is small enough for the stated exponent.

The tower starts from the quintic field and its quadratic extension
of Hajir--Maire \cite[\S3.2]{HM}. The simultaneous Frobenius cuts are
obtained by the method of Hajir--Maire--Ramakrishna
\cite[\S\S1.1--1.4 and \S2.1]{HMR}. Since the resulting real fields
need not be normal over $\Q$, we compare prime decomposition in the
normal tower with that in the fixed field of complex conjugation.
The signed Euler terms are retained in the Tsfasman--Vl\u adu\c t
estimates \cite{TV}.

All terminating decimals used as parameters or numerical bounds are
exact rational numbers. The finite numerical inequalities are established
by outward rational interval arithmetic, as described in
Appendix~
\ref{app:intervals}; the geometric and tower arguments are given
in the text.

\subsection{Conventions}
All logarithms are natural. For a number field $L$, $d=[L:\Q]$,
$\Delta_L$ is its absolute discriminant, $\rd(L)=|\Delta_L|^{1/d}$,
$h_L$ is its class number, and $R_L$ is the usual regulator using
weight two at complex places. We give each complex coordinate in
Minkowski space its ordinary planar Lebesgue measure. The integer
lattice then has covolume $2^{-r_2}\sqrt{|\Delta_L|}$.
The norm of an ideal is its absolute norm. For $x\in\R$, put
$x_+=\max\{x,0\}$. Constants and the number of prime families are
fixed before the degrees tend to infinity, even when that number
is large.

\section{A Frobenius-controlled tower and its real subfields}
\label{sec:arithmetic}

We use the quintic field of Hajir and Maire \cite[\S3.2]{HM}:
\begin{equation}\label{arith:base}
 k=\Q(\xi),\quad
 f_0(\xi)=0,\quad
 f_0(x)=x^5-2x^4+3x^3-3x^2-x+1.
\end{equation}
Put
\begin{equation}\label{arith:eta}
 \eta=-36\xi^4+125\xi^3-221\xi^2+182\xi-80,
 \qquad F=k(\sqrt\eta).
\end{equation}
Throughout this section set
\begin{equation}\label{arith:constants}
 C=16800001,\quad Y=70001,\quad
 \delta=84.37558833,\quad c_*=\frac{301}{10000}.
\end{equation}
Terminating decimals denote exact rational numbers.

\begin{lemma}\label{arith:model}
The field $k$ has signature $(3,1)$, discriminant $-31391$ and
$\OO_k=\mathbb Z[\xi]$. The field $F$ is totally imaginary of degree $10$,
and
\[
 |\Delta_F|=31391^2\cdot18559139599,\qquad
 18559139599=7^2 11^2 13\,19^2 23\,29.
\]
Exactly nine finite primes of $k$ ramify in $F/k$. They are
\[
 \begin{array}{c|rrrrrr}
 p&7&11&13&19&23&29\\ \hline
 \mathfrak p=(p,\xi-r),\ r&3,6&2,7&5&3,6&9&18 .
 \end{array}
\]
Each has residue degree one and occurs with valuation one in $(\eta)$.
The unique prime of $k$ of norm $32$ splits in $F/k$, as do
\[
 (43,\xi-6),\ (59,\xi-17),\ (61,\xi-56),\
 (83,\xi-63),\ (97,\xi-95),\ (131,\xi-21).
\]
\end{lemma}
\begin{proof}
The reduction of $f_0$ modulo $2$ is irreducible: in its quotient ring
$x^{32}=x$ and $\gcd(f_0,x^2-x)=1$. Since $5$ is prime, these two
identities are the irreducibility criterion. A Sylvester determinant gives
$\operatorname{disc}(f_0)=-31391$, and $31391$ is prime. Thus the integral
basis is $1,\xi,\ldots,\xi^4$. Sturm's theorem gives three real roots.
Rational isolation followed by interval Horner evaluation shows that
$\eta$ is negative at each of them.

The norm determinant gives $\NN_{k/\Q}(\eta)=-18559139599$.
For the six rational primes in the table, the displayed roots are exactly
the common roots of $f_0$ and $\eta$ in the prime field. Their prime norms
already multiply to $|\NN_{k/\Q}(\eta)|$. Consequently the displayed
factorization of $(\eta)$ is squarefree and exhaustive.

Let
\[
 \beta=1+\xi+\xi^4,\qquad
 \gamma_0=20-45\xi+56\xi^2-31\xi^3+11\xi^4.
\]
Reduction modulo $f_0$ gives $\eta=\beta^2-4\gamma_0$. Hence
$\theta=(\beta+\sqrt\eta)/2$ is integral, with equation
$\theta^2-\beta\theta+\gamma_0=0$.
The order $\OO_k[\theta]$ has relative discriminant $(\eta)$.
The square of its index ideal divides this squarefree ideal, so the
order is maximal. The relative discriminant formula proves the value of
$\Delta_F$. Negativity at all three real places proves the assertion
about the signature.

For each odd prime in the final list, substitution of the indicated root
shows that $\eta$ has a nonzero square residue. At $2$, write
$\bar\xi$ for the class of $x$ in $\mathbb F_{32}$. One computes
\[
 \bar\beta^{-1}=1+\bar\xi^2,\qquad
 \bar\gamma_0/\bar\beta^2=1+\bar\xi+\bar\xi^2+\bar\xi^3.
\]
The latter element has trace zero to $\mathbb F_2$, so the reduced
quadratic equation splits. All computations just described use exact
integer and finite-field arithmetic.
\end{proof}

\begin{proposition}\label{arith:tower}
There exists an infinite extension $K_\infty/F$, unramified at all finite
places, pro-$2$ over $F$, and Galois over $k$, with the following
Frobenius conditions. At the unique $F$-primes over
$(7,\xi-3),(7,\xi-6),(11,\xi-2)$ the Frobenius order divides $4$.
At the other six ramified $k$-primes and at all fourteen $F$-primes
over the seven split $k$-primes in Lemma~\ref{arith:model}, it divides $8$.
At every $F$-prime over a $k$-prime $\mathfrak p$ with
$\NN\mathfrak p<C$, it divides $32$.
\end{proposition}
\begin{proof}
Let $G$ be the Galois group of the maximal everywhere unramified
pro-$2$ extension of $F$, and let $D$ and $r$ be its generator and
relation ranks. Genus theory gives $D\ge7$:
the extension $F/k$ has nine finite and three real ramified places,
and the bound in \cite[Theorem~2.5]{HM}, with the four archimedean
places of $k$, is $12-4-1=7$. The Shafarevich bound
\cite[Theorem~2.2]{HM} gives
\[
 r-D\le |S_\infty(F)|-1+1=5.
\]
The final $1$ is essential: $F$ contains $\mu_2$ and the finite
ramification set is empty.

For each prime in the statement impose the indicated power of a
Frobenius element, taking the closed normal closure of these relations.
The complete set of primes is stable under $\Gal(F/k)$: the ramified
primes are fixed, the split pairs are included together, and the
norm cutoff is imposed on $k$-primes. Thus this closed normal subgroup
is also normal in the Galois group over $k$.

There are at most $5C$ primes of $k$ of norm less than $C$:
each lies above a rational prime less than $C$, and there are at most
five primes above each such rational prime. Each has at most two
extensions to $F$, giving at most $10C$ further relations.
In a free pro-$2$ presentation a $2^a$-th power has Zassenhaus depth
at least $2^a$ \cite[\S1.1 and \S1.4]{HMR}. The original minimal
relations have depth at least two. Therefore, for $0<t<1$, an upper
bound for the Golod--Shafarevich polynomial of the quotient is
\begin{align*}
 1-Dt+rt^2+3t^4+20t^8+10Ct^{32}
 &\le 1-7t+12t^2+3t^4+20t^8+10Ct^{32}.
\end{align*}
Here the difference in the degree-two part is at most
$(D-7)(t^2-t)\le0$, so equality of the generator rank with $7$
is not assumed.
At $t=7/25$ the last expression is
\[
 -\frac{511272996594583443171586679224792442223}
 {108420217248550443400745280086994171142578125}<0.
\]
The Golod--Shafarevich criterion \cite[\S1.2]{HMR} proves infinitude.
Taking a quotient adds no ramification, and the stated orders follow
from the imposed relations.
\end{proof}

Fix a real embedding of $k$ and an algebraic closure in $\C$.
Complex conjugation induces an involution $c$ of $K_\infty/k$,
nontrivial already on $F$. Choose nested finite subextensions
$K_j/k$, Galois over $k$, containing $F$, and with degrees tending
to infinity, and set $L_j=K_j^{\langle c\rangle}$.
Then
\begin{equation}\label{arith:descent}
 L_j\subset\R,\qquad d_j=[L_j:\Q]=\tfrac12[K_j:\Q]\longrightarrow\infty,
 \qquad \rd(L_j)\le|\Delta_F|^{1/10}<\delta.
\end{equation}
Moreover $L_j$ contains $k$, and each real embedding restricts to one
of its three real embeddings. Thus
\begin{equation}\label{arith:signature}
 0<r_1(L_j)/d_j\le3/5.
\end{equation}
For the finite set of $k$-primes of norm less than $C$, the residue
degrees in the nested Galois layers $K_j/F$ eventually stabilize:
they are nondecreasing under divisibility and divide $32$. A split
pair has a common residue degree, since $K_j/k$ is Galois.
Discard finitely many layers so that these degrees are constant.

We index a family by each $k$-prime $\mathfrak p$ of norm $q<C$.
Write $f_\mathfrak p$ for this stabilized degree over $F$ and define
\begin{equation}\label{arith:families}
 (Q_\mathfrak p,a_\mathfrak p)=
 \begin{cases}
 (q,2),&\mathfrak p\text{ splits in }F/k,\\
 (q,1),&\mathfrak p\text{ ramifies in }F/k,\\
 (q^2,1),&\mathfrak p\text{ is inert in }F/k.
 \end{cases}
\end{equation}
The contraction to $L_j$ of the $K_j$-primes in a family produces at
least $a_\mathfrak p d_j/(10f_\mathfrak p)$ distinct $L_j$-primes,
all of norm at most $Q_\mathfrak p^{f_\mathfrak p}$.
Indeed $K_j/F$ is unramified and has
$a_\mathfrak p[K_j:F]/f_\mathfrak p$ primes in this family, while
at most two primes contract to one prime of $L_j$.
Families indexed by different $k$-primes are disjoint.
In particular a split pair must be treated as one family, since
its two parts can have the same contraction.
We may select any prescribed number of these primes up to this
lower bound.

\section{A signed Euler bound for the class number and regulator}
\label{sec:regulator}

For a prime power $u$, let $N_u(L)$ denote the number of finite primes
of $L$ having absolute norm $u$.
After a diagonal subsequence, assume that the root discriminants,
signature proportions, and all degree-normalized prime counts converge.
Put
\[
 \delta_L=\lim_j\rd(L_j),\quad
 \mu=\lim_jr_1(L_j)/d_j,\quad
 \nu=(1-\mu)/2,\quad
 x_u=\lim_jN_u(L_j)/d_j .
\]
The root discriminants are bounded below by $\rd(k)>1$, so this
degree normalization is equivalent to genus normalization.
For every integer $u\ge2$ define
\[
 a(u)=2\log u\sum_{n\ge1}(u^n+1)^{-1},\qquad
 b(u)=\log\frac{u}{u-1},\qquad
 V(u)=b(u)-c_*a(u).
\]
Let $z_+=\max(z,0)$.
The unconditional basic inequality and generalized Brauer--Siegel
upper bound \cite[Proposition~3.1 and Theorem~7.1]{TV} give
\begin{align}
 \sum_u x_u a(u)
 &\le\tfrac12\log\delta_L-\tfrac12(\gamma_E+\log(4\pi))
              -\tfrac12\mu,\label{arith:basic}\\
 \limsup_j\frac{\log(h_{L_j}R_{L_j})}{d_j}
 &\le\tfrac12\log\delta_L-\mu\log2-\nu\log(2\pi)
                  +\sum_u x_u b(u).\label{arith:BS}
\end{align}
The sums indexed by $u$ run over prime powers. No equality form of
Brauer--Siegel, normality over $\Q$, or GRH is used.
The sums converge: $b(u)/a(u)$ is bounded, and
$|V(u)|\le b(u)+c_*a(u)$.

\subsection{The permutation comparison}
For one layer put
\[
 \mathcal G=\Gal(K_j/k),\quad H=\langle c\rangle,\quad
 h=\frac{|C_{\mathcal G}(c)|}{|\mathcal G|}.
\]
A coset $gH$ is fixed by $c$ precisely when $g^{-1}cg=c$.
Above the fixed real embedding of $k$ there are
$|C_{\mathcal G}(c)|/2$ real embeddings of $L_j$, while
$d_j=5|\mathcal G|/2$. Consequently
\begin{equation}\label{arith:centralizer}
 h\le5\,r_1(L_j)/d_j.
\end{equation}

Let $\mathfrak p$ be a $k$-prime of norm $q$, with decomposition and
inertia groups $\mathcal D,\mathcal I$ in $\mathcal G$, and write
$e=|\mathcal I|$, $f'=|\mathcal D/\mathcal I|$.
For a coset $gH$ the stabilizer
$J=\mathcal D\cap gHg^{-1}$ is trivial or has order two.
The corresponding ramification index and residue degree in $L_j/k$
are
\[
 \frac e{|\mathcal I\cap J|},\qquad
 \frac{f'|\mathcal I\cap J|}{|J|}.
\]
Let $p_R$ and $p_I$ be the proportions of cosets for which
$J\ne1$ lies outside or inside $\mathcal I$, respectively.
Thus the three possibilities have $(e_L,f_L)$ equal to
$(e,f'),(e,f'/2),(e/2,f')$.
For any real weight $W$, including a signed weight, the contribution
of this $k$-prime to the degree-normalized prime sum in $L_j$ is
\begin{equation}\label{arith:orbit}
 \frac{(1-p_R-p_I)W(q^{f'})
       +2p_RW(q^{f'/2})+2p_IW(q^{f'})}{5ef'}.
\end{equation}
The middle term is absent when $f'=1$.
To verify the coefficient, an orbit with ramification index $e_L$
and residue degree $f_L$ contains $e_Lf_L$ embeddings, so its
prime count per $[L_j:k]$ is the corresponding coset proportion
divided by $e_Lf_L$.

Every conjugate of $c$ fixes a fraction $h$ of the cosets.
Fixed sets for two distinct conjugates are disjoint because a coset
stabilizer has order two. It follows that
\[
 p_R=h\,\#\{x\in\mathcal D\setminus\mathcal I:
                   x\text{ is conjugate to }c\},\quad
 p_I=h\,\#\{x\in\mathcal I:x\text{ is conjugate to }c\}.
\]
Since $\mathcal D/\mathcal I$ is cyclic, its order-two element
has at most $e$ lifts, giving $p_R\le eh$.
Moreover, $\Gal(K_j/F)$ is a normal subgroup of index two in
$\mathcal G$, and $c$ has nontrivial image in the quotient. Hence
every conjugate of $c$ also has nontrivial image there, so no
conjugate of $c$ belongs to $\Gal(K_j/F)$.

If $\mathfrak p$ splits in $F/k$, its decomposition group is contained
in $\Gal(K_j/F)$, so $p_R=p_I=0$.
Its contribution for $W=V$ is exactly
$2V(q^{f_\mathfrak p})/(10f_\mathfrak p)$.
If it is inert, $e=1$, $f'=2f_\mathfrak p$ and $p_I=0$.
Its contribution is at most
\[
 \frac{V(q^{2f_\mathfrak p})}{10f_\mathfrak p}
       +\frac{r_1(L_j)}{d_j}\,
               \mathcal C^{\mathrm{in}}_{q,f_\mathfrak p},
 \qquad
 \mathcal C^{\mathrm{in}}_{q,f}
   =\frac{(2V(q^f)-V(q^{2f}))_+}{2f}.
\]
Finally, if it ramifies in $F/k$, then $e=2$, $f'=f_\mathfrak p$,
$p_I\le h$ and $p_R\le2h$. Recall that $f_\mathfrak p\mid32$;
hence $f_\mathfrak p>1$ implies that $f_\mathfrak p$ is even.
Its contribution is at most
\[
 \frac{V(q^{f_\mathfrak p})}{10f_\mathfrak p}
       +\frac{r_1(L_j)}{d_j}\,
                \mathcal C^{\mathrm{ram}}_{q,f_\mathfrak p},
\]
where
\begin{equation}\label{arith:ramcorrection}
 \mathcal C^{\mathrm{ram}}_{q,1}=\tfrac12V(q)_+,\qquad
 \mathcal C^{\mathrm{ram}}_{q,f}
 =\frac{2(2V(q^{f/2})-V(q^f))_++V(q^f)_+}{2f}
 \quad(f>1).
\end{equation}
These comparisons preserve the signed baseline $V(Q^f)$.
Only the excess caused by a nontrivial stabilizer is replaced by
its positive part.

\subsection{The cutoff and the joint bound}
The elementary estimates
\[
 b(u)\le\frac1{u-1},\qquad
 \frac{2\log u}{u+1}\le a(u)\le\frac{2\log u}{u-1}
\]
show that
\begin{equation}\label{arith:cutoff}
 V(u)<0,\qquad 2V(u)-V(u^2)<0\qquad(u\ge C).
\end{equation}
For the second assertion an upper bound is
\[
 \frac{2}{u-1}
  \left(1-2c_*\log u\,\frac{u-2}{u+1}\right).
\]
A direct rational interval computation gives
$2c_*\log C>(C+1)/(C-2)$, and the product
$\log u\,(u-2)/(u+1)$ is increasing for $u\ge2$.
The first assertion follows from the weaker threshold
$(u+1)/(u-1)$.

For each family of \eqref{arith:families} set
\[
 v_\mathfrak p=\frac{a_\mathfrak p}{10f_\mathfrak p}
                    V(Q_\mathfrak p^{f_\mathfrak p}),\qquad
 \mathcal C_\mathfrak p=
 \begin{cases}
 0,&\mathfrak p\text{ splits},\\
 \mathcal C^{\mathrm{in}}_{q,f_\mathfrak p},&\mathfrak p\text{ is inert},\\
 \mathcal C^{\mathrm{ram}}_{q,f_\mathfrak p},&\mathfrak p\text{ ramifies}.
 \end{cases}
\]
All sums in the following proposition are finite.

\begin{proposition}\label{arith:hR}
Put
\[
 U=\frac12\log\frac{\delta}{4\pi}
       +\frac{c_*}{2}\bigl(\log\delta-\gamma_E-\log(4\pi)\bigr).
\]
Then $U<0.972098863$, and
\begin{equation}\label{arith:hRbound}
 \limsup_j\frac{\log(h_{L_j}R_{L_j}/2^{r_2(L_j)})}{d_j}
 \le U+\sum_{\NN\mathfrak p<C}v_\mathfrak p
       +\mu\left(\tfrac12\log\pi-\tfrac12c_*
                            +\sum_{\NN\mathfrak p<C}\mathcal C_\mathfrak p\right).
\end{equation}
\end{proposition}
\begin{proof}
Every $L_j$-prime above a $k$-prime of norm at least $C$ has norm
at least $C$, so its contribution for the weight $V$ is nonpositive.
The remaining set of $k$-primes is finite. The orbit comparison can
therefore be summed and passed to the limit, giving
\[
 \sum_u x_uV(u)\le
 \sum_{\NN\mathfrak p<C}v_\mathfrak p+
                 \mu\sum_{\NN\mathfrak p<C}\mathcal C_\mathfrak p.
\]
Here it is legitimate to retain negative values: these are exact
baseline contributions at the finitely many controlled primes.
Using $b=c_*a+V$ in \eqref{arith:BS}, then
\eqref{arith:basic}, and subtracting $\nu\log2$, gives
\eqref{arith:hRbound}. The right side before replacing $\delta_L$
by $\delta$ is increasing in $\log\delta_L$, with coefficient
$(1+c_*)/2>0$, so \eqref{arith:descent} justifies that replacement.
\end{proof}

For later use, if $\Lambda_L$ is the lattice of unweighted unit
logarithms in $\sum_i\kappa_i x_i=0$, with $\kappa_i=1$ at real
places and $2$ at complex places, then
\[
 \covol(\Lambda_L)
   =\frac{\sqrt{r_1+4r_2}}{2^{r_2}}R_L.
\]
Indeed, deleting row $i$ from an unweighted logarithm matrix gives
a minor of absolute value $\kappa_iR_L/2^{r_2}$, and
Cauchy--Binet gives the displayed covolume.
For an $S$-unit assignment lattice the class-group index $J_S$
is at most $h_L$, so the same bound applies to
$J_S\covol(\Lambda_L)$ up to the factor
$\sqrt{r_1+4r_2}=\exp(o(d))$.

\subsection{A reproducible bound for inert corrections}
\label{arith:tail}
The finitely many ramified corrections are retained individually.
It remains to obtain a uniform bound for the inert corrections with
$\NN\mathfrak p<C$. For a safe bound on all possible residue degrees put
\[
 I_{\mathrm{pre}}=
 \sum_{\substack{\mathfrak p\ {\rm inert}\\\NN\mathfrak p<Y}}
       \max_{f\in\{1,2,4,8,16,32\}}\mathcal C^{\mathrm{in}}_{\NN\mathfrak p,f}.
\]
Exact finite-field factorization gives
\begin{equation}\label{arith:prefixbound}
 I_{\mathrm{pre}}<0.317565174395.
\end{equation}
The following explicit tail calculation bounds the remaining primes.
Define
\[
 J(u)=\frac1{u-1}-\frac{2c_*\log u}{u+1}
                         +\frac{2c_*\log u}{u^2-1}.
\]
Then $\mathcal C^{\mathrm{in}}_{q,f}\le J(q^f)_+/f$.
Writing
$J(u)=[1-2c_*\log u\,(u-2)/(u+1)]/(u-1)$ shows that
$J(u)_+$ is nonincreasing on $[2,\infty)$.
Thus every correction for $q\ge Y$ is at most $J(q)_+$.

For rational primes $p\ge Y$, every $k$-prime above $p$ has inert
correction at most $J(p)_+$, and there are at most five such $k$-primes.
Divide $[Y,C)$ into the eight intervals $[L,R)$ with
$L=2^iY$ and $R=\min(2L,C)$.
On each interval, $2c_*\log L<1$ and
\[
 J(p)_+\le\frac1{p-1}-\frac{2c_*\log L}{p+1}
                         +\frac{2c_*\log R}{p^2-1}.
\]
The expression on the right is positive. A sieve enumerates the
rational primes in each interval. With $s=10^{24}$ define
\[
 A=\frac1s\sum_{L\le p<R}\left\lceil\frac{s}{p-1}\right\rceil,\quad
 B=\frac1s\sum_{L\le p<R}\left\lfloor\frac{s}{p+1}\right\rfloor,\quad
 D=\frac1s\sum_{L\le p<R}\left\lceil\frac{s}{p^2-1}\right\rceil.
\]
Using lower and upper rational logarithm bounds $\ell_L$ and $u_R$
gives the bucket upper bound $A-2c_*\ell_LB+2c_*u_RD$.
There are at most five $k$-primes above a rational prime.

One must also include higher-degree primes with rational
$p<Y$ but norm $p^r\ge Y$, for $2\le r\le5$.
Let $n_r$ be the least integer with $n_r^r\ge Y$, and put
\[
 S_{r,a}=n_r^{-a}+\frac{n_r^{1-a}}{a-1}\qquad(a>1).
\]
Dropping the negative term of $J$ and using $\log(q^f)<\log C$
when the correction is nonzero gives the following upper bound for
this part:
\[
 5\sum_{r=2}^5
  \left(\frac{Y}{Y-1}S_{r,r}+
      2c_*\log C\,\frac{Y^2}{Y^2-1}S_{r,2r}\right).
\]
The elementary integral test
$\sum_{n\ge n_r}n^{-a}\le n_r^{-a}+n_r^{1-a}/(a-1)$
justifies this formula. Summing both tail parts with outward rational
rounding gives
\begin{equation}\label{arith:tailbound}
 I_{\mathrm{tail}}<0.410327590,\qquad
 I_{\mathrm{all}}:=I_{\mathrm{pre}}+I_{\mathrm{tail}}
                  <0.727892765.
\end{equation}
This deliberately overcounts primes, and hence remains valid for
the inert subcollection. All numerical endpoints in these estimates
are understood as outward-rounded rational bounds.

\section{A logarithmic volume bound and a sharp two-point moment}
\label{sec:unit-body}
Let $\kappa_i\in\{1,2\}$, $\sum_i\kappa_i=d$, and let $n$ be the
number of coordinates. Put
\[
 \mathcal H=\{x\in\R^n:\textstyle\sum_i\kappa_i x_i=0\}.
\]
We give this hyperplane its Euclidean $(n-1)$-dimensional measure.
The following entropy bound is the archimedean input to the
mixed $S$-unit construction.

Let $\psi=\Gamma'/\Gamma$ denote the digamma function. For $\alpha>0$, set
\[
 B_\alpha=\alpha e^{-\psi(\alpha)},\qquad
 H_\alpha=e^\alpha\Gamma(\alpha)e^{-\alpha\psi(\alpha)}.
\]
Fix $0<\tau\le1$ and $M>B_\alpha$, and define
\begin{equation}\label{geom:unitbody}
 C_{\alpha,\tau,M}=\left\{x\in\mathcal H:
 \sum_i\kappa_i e^{\tau x_i}\le dM\right\}.
\end{equation}
This is a compact convex body in $\mathcal H$. Indeed, the defining function
is convex, each coordinate has an upper bound, and the equation defining
$\mathcal H$ then bounds each coordinate below. Also $B_\alpha>1$, as follows
from strict Jensen's inequality applied to the random variable used in the
next proof. Hence the origin is an interior point relative to $\mathcal H$.

\begin{lemma}\label{geom:entropy}
For fixed $\alpha,\tau,M$ as above,
\[
 \log\vol_{n-1}(C_{\alpha,\tau,M})
 \ge n\log(H_\alpha/\tau)-o(d).
\]
\end{lemma}
\begin{proof}
Put $a_\alpha=e^{\psi(\alpha)}$, and let $Y$ have the density
\[
 h_\alpha(y)=\frac{\tau a_\alpha^\alpha}{\Gamma(\alpha)}
 \exp\{\alpha\tau y-a_\alpha e^{\tau y}\},\qquad y\in\R.
\]
The change of variables $z=a_\alpha e^{\tau y}$ identifies $z$ with a
Gamma random variable of shape $\alpha$ and scale $1$. Consequently
\[
 \E Y=0,\qquad \E e^{\tau Y}=B_\alpha,
 \qquad -\E\log h_\alpha(Y)=\log(H_\alpha/\tau).
\]
All three variables $Y$, $e^{\tau Y}$ and $\log h_\alpha(Y)$ have finite
variance.

Take independent copies $Y_1,\ldots,Y_n$. For any sufficiently small fixed
$\zeta>0$, Chebyshev's inequality shows that the following conditions hold
on a set of probability $1-o(1)$:
\[
 |t|\le\zeta,\qquad t=\frac1d\sum_i\kappa_iY_i,\qquad
 \frac1d\sum_i\kappa_i e^{\tau Y_i}\le B_\alpha+\zeta,
\]
and
\[
 \prod_i h_\alpha(Y_i)
 \le\exp\{-n(\log(H_\alpha/\tau)-\zeta)\}.
\]
For the weighted assertions one uses
$\sum_i\kappa_i^2\le2d$. The corresponding typical set in $\R^n$ therefore
has volume at least
\[
 (1-o(1))\exp\{n(\log(H_\alpha/\tau)-\zeta)\}.
\]
For each point $y$ in this set, put $x=y-t\mathbf 1$. Then $x\in\mathcal H$,
and
\[
 \sum_i\kappa_i e^{\tau x_i}
 \le de^{\tau\zeta}(B_\alpha+\zeta)\le dM
\]
if $\zeta$ is small enough. The map $(x,t)\mapsto x+t\mathbf 1$, with
$x\in\mathcal H$, has Jacobian
$d/\|(\kappa_i)\|$. Since $|t|\le\zeta$, the volume of the typical set is at
most
\[
 2\zeta\,\frac d{\|(\kappa_i)\|}
 \vol_{n-1}(C_{\alpha,\tau,M}).
\]
The logarithm of this factor is $o(d)$. Letting first $d$ tend to infinity
and then $\zeta$ tend to zero proves the assertion.
\end{proof}

\subsection{A two-point moment inequality}\label{geom:momentsection}

A norm-one constraint allows a sharp improvement over applying
subadditivity separately to two vectors. Put
\begin{equation}\label{geom:F}
 F_\tau(M)=\left[(M+\sqrt{M^2-1})^{1/\tau}
 +(M-\sqrt{M^2-1})^{1/\tau}\right]^\tau,
 \qquad M\ge1.
\end{equation}

\begin{lemma}\label{geom:moment}
Suppose $a_i,b_i>0$ and
\[
 \prod_i a_i^{\kappa_i}=\prod_i b_i^{\kappa_i}=1,
 \qquad \sum_i\kappa_i a_i\le dM,
 \qquad \sum_i\kappa_i b_i\le dM.
\]
Then, for $0<\tau\le1$,
\[
 \sum_i\kappa_i(a_i^{1/\tau}+b_i^{1/\tau})^\tau
 \le dF_\tau(M).
\]
The bound is attained when there are an even number of equally weighted
coordinates. In addition, $F_\tau(M)\le2M$.
\end{lemma}
\begin{proof}
For any family $(x_i)$ write
\[
 \bar x:=\frac1d\sum_i\kappa_i x_i .
\]
Thus all averages below use the weights $\kappa_i/d$. If $M=1$,
equality in the arithmetic--geometric mean inequality forces every $a_i$ and $b_i$ to equal
$1$, and the claim follows. If $\tau=1$, the assertion is just the sum of the
two arithmetic-mean bounds, since $F_1(M)=2M$. We may therefore assume
$M>1$ and $0<\tau<1$.

Write
\[
 r_i=\sqrt{a_ib_i},\qquad t_i=\frac{a_i+b_i}{2r_i}\ge1,
 \qquad
 Q(t)=2^\tau\cosh\bigl(\tau^{-1}\operatorname{arcosh}t\bigr)^\tau.
\]
The function $Q$ is increasing and concave on $[1,\infty)$.
Indeed, for $t=\cosh v>1$ and $p=1/\tau>1$,
\[
 \frac{Q'(t)}{Q(t)}=\frac{\tanh(pv)}{\sinh v}>0.
\]
For concavity, at $t>1$ the inequality $Q''(t)\le0$ is equivalent to
\[
 \sinh(pv)\cosh((p-1)v)\ge p\sinh v.
\]
This follows from
\[
 2\sinh(pv)\cosh((p-1)v)
 =\sinh((2p-1)v)+\sinh v\ge2p\sinh v,
\]
using convexity of $\sinh$ on $[0,\infty)$. The extension at $t=1$ is
continuous and has $Q'(1)=2^\tau/\tau$.

The tangent line to $Q$ at $M$ has the form $ct-k$, where $c=Q'(M)>0$ and
$k=MQ'(M)-Q(M)\ge0$. The last inequality follows, on writing
$M=\cosh v$, from $\tanh(pv)\coth v\ge1$. Since
\[
 \overline r\ge1,\qquad \overline{rt}
 =\frac12(\overline a+\overline b)\le M,
\]
we obtain
\[
 \overline{rQ(t)}\le c\overline{rt}-k\overline r
 \le cM-k=Q(M)=F_\tau(M).
\]
The left side is the required moment divided by $d$.

For the equality example, take $a=M+\sqrt{M^2-1}$ and
$b=M-\sqrt{M^2-1}$ on half the equally weighted coordinates, and interchange
$a,b$ on the other half. Finally, subadditivity of $x\mapsto x^\tau$ gives
$F_\tau(M)\le M+\sqrt{M^2-1}+M-\sqrt{M^2-1}=2M$.
\end{proof}

\section{The outer parallel body}
\label{subsec:vertex-packing}

We use Lebesgue measure on the Minkowski space
$\R^{r_1}\times\C^{r_2}$, with ordinary planar measure on each complex
coordinate. Let $\kappa_i=1$ on a real coordinate and $\kappa_i=2$
on a complex coordinate, so that $\sum_i\kappa_i=d$.
For $0<\tau\le1$, $S>0$, and a collection $I$ of coordinates, put
\[
 \mathcal B_{\tau,I}(Y)
 =\left\{z:\sum_{i\in I}\kappa_i|z_i|^\tau\le Y\right\},
 \qquad
 \mathcal Q_I=\mathcal B_{1,I}(d/2).
\]
An empty-coordinate body has volume $1$.
We omit the subscript when all coordinates are present.

\begin{lemma}[Outer parallel volume]\label{lem:vertex-packing}\label{geom:outer-parallel}
Fix $\tau\in(0,1]$ and $S>0$. Set
\begin{align*}
 A_1&=\frac{\log(S\tau)}{\tau}
          +\log\bigl(2\Gamma(1+1/\tau)\bigr),\\
 A_2&=\frac{\log(S\tau/2)}{\tau}
          +\frac12\log\Gamma(1+2/\tau)+\frac32\log2,\\
 E_\tau(S)&=2\exp\left(-\frac12\min\{A_1,A_2\}\right).
\end{align*}
As $d=r_1+2r_2\to\infty$,
\begin{equation}\label{eq:vertex-volume}
 \log\vol\bigl(\mathcal B_\tau(dS)+\mathcal Q\bigr)
 \le \log\vol\mathcal B_\tau(dS)+dE_\tau(S)+O_{\tau,S}(\log(d+2)).
\end{equation}
The implied constant is independent of the signature.
\end{lemma}

\begin{proof}
Write a point of the Minkowski sum as $z=u+v$, and put $r_i=|z_i|$.
The triangle inequality implies that
$q_i=(r_i-|u_i|)_+$ satisfies
\[
 0\le q_i\le r_i,\qquad
 \sum_i\kappa_iq_i\le d/2,\qquad
 \sum_i\kappa_i(r_i-q_i)^\tau\le dS.
\]
Minimize the last, concave function on the compact polytope
\[
 \mathcal P_r=\left\{q:0\le q_i\le r_i,\
                              \sum_i\kappa_iq_i\le d/2\right\}.
\]
A minimum occurs at a vertex: expressing any point as a convex
combination of vertices and applying concavity proves this assertion.
Since the feasible $q$ obtained above has objective at most $dS$, a
minimizing vertex also has objective at most $dS$.
A vertex has at most one coordinate strictly between $0$ and $r_i$.
Indeed, if two coordinates were interior, sufficiently small opposite
weighted perturbations would preserve the budget and all box constraints.
This would express the point as the midpoint of two distinct points of
the polytope. If the budget is inactive, a single interior coordinate
already contradicts extremality.

Let $J$ consist of the coordinates with $q_i=r_i$, and let $\ell$ be
the possible exceptional coordinate. Then
\[
 \sum_{i\in J}\kappa_i r_i\le d/2,\qquad
 \sum_{i\notin J\cup\{\ell\}}\kappa_i r_i^\tau\le dS.
\]
For the exceptional coordinate,
\[
 r_\ell\le (dS/\kappa_\ell)^{1/\tau}+d/(2\kappa_\ell).
\]
We may also designate an arbitrary exceptional coordinate when all
coordinates are at endpoints; deleting that coordinate from either
budget preserves both inequalities. Thus the parallel body is covered
by a union, over $\ell$ and $J$, of products of a projected
$\ell^1$ body, a projected $\ell^\tau$ body, and one real interval or
complex disc of radius $O_{\tau,S}(d^{1/\tau}+d)$.
The last factor has logarithmic volume $O_{\tau,S}(\log(d+2))$.
Allowing arbitrary arguments in the complex coordinates yields radial
supersets and is valid for an upper bound.

If $D_I=\sum_{i\in I}\kappa_i$, polar integration and the Dirichlet
integral give the exact formula
\begin{equation}\label{eq:projected-volume}
 \vol\mathcal B_{\tau,I}(Y)
 =\frac{Y^{D_I/\tau}}{\Gamma(1+D_I/\tau)}
       \prod_{i\in I}
       \frac{v_{\kappa_i}\Gamma(1+\kappa_i/\tau)}
            {\kappa_i^{\kappa_i/\tau}},
 \qquad v_1=2,\quad v_2=\pi .
\end{equation}
For $D_I=0$ it is interpreted as $1$.
Stirling's estimate has error $O_\tau(\log(d+2))$ uniformly over
the integers $0\le D_I\le d$; the zero case is taken separately.
Changing either projected dimension by at most $2$ changes the resulting
main term by $O_{\tau,S}(\log(d+2))$. Consequently the exceptional
coordinate can be omitted from the following exponential-rate calculation.

Put $\mu=r_1/d$. If $J$ contains a proportion $\theta_1$ of the real
coordinates and a proportion $\theta_2$ of the complex coordinates, put
\[
 \rho=\mu\theta_1+(1-\mu)\theta_2,\qquad
 h_{\rm bin}(x)=-x\log x-(1-x)\log(1-x).
\]
We use the continuous convention $0\log0=0$, and all expressions
involving $\rho$ or $1-\rho$ at the endpoints are interpreted by
continuous extension.
The number of choices of $J$ with these proportions has logarithmic rate
at most
$\mu h_{\rm bin}(\theta_1)+(1-\mu)h_{\rm bin}(\theta_2)/2$.
This follows directly from
$\binom nk\le\exp(nh_{\rm bin}(k/n))$; no nonuniform endpoint
approximation is required.
Define
\[
 b_1=\log\bigl(2\Gamma(1+1/\tau)\bigr),\quad
 b_2=\frac12\log\bigl(\pi\Gamma(1+2/\tau)\bigr)
                                 -\frac{\log2}{\tau},
 \qquad
 q_1=\log2,\quad q_2=\frac12\log(\pi/2).
\]
Formula \eqref{eq:projected-volume} bounds the logarithmic volume of the
corresponding union stratum, divided by $d$, by
\begin{align*}
 V_\mu(\theta_1,\theta_2)
 ={}&\frac{1-\rho}{\tau}
       \left(\log\frac{S\tau}{1-\rho}+1\right)
    +\mu(1-\theta_1)b_1+(1-\mu)(1-\theta_2)b_2\\
 &+\rho\left(\log\frac1{2\rho}+1\right)
    +\mu\theta_1q_1+(1-\mu)\theta_2q_2\\
 &+\mu h_{\rm bin}(\theta_1)
                   +\frac{1-\mu}{2}h_{\rm bin}(\theta_2)
    +O_{\tau,S}\!\left(\frac{\log(d+2)}d\right).
\end{align*}
There are at most $(r_1+1)(r_2+1)d$ strata, including choices of the
exceptional coordinate, so summing their volumes contributes only another
$O(\log(d+2))$ to the logarithm.
The volume rate of the original body is
\[
 V_0=\frac{\log(S\tau)+1}{\tau}+\mu b_1+(1-\mu)b_2.
\]

Let $M_*=\max\{q_1-b_1,q_2-b_2\}$.
Concavity and nonnegativity of binary entropy show that
\[
 \mu h_{\rm bin}(\theta_1)
       +\frac{1-\mu}{2}h_{\rm bin}(\theta_2)
 \le \mu h_{\rm bin}(\theta_1)+(1-\mu)h_{\rm bin}(\theta_2)
 \le h_{\rm bin}(\rho).
\]
The local-constant change is at most $\rho M_*$. Using
$-(1-\rho)\log(1-\rho)\le\rho$, we obtain
\[
 V_\mu(\theta_1,\theta_2)-V_0
 \le \rho(2-A)-2\rho\log\rho,\qquad
 A=\frac{\log(S\tau)}{\tau}+\log2-M_*.
\]
Substitution of the two possible values in $M_*$ gives
$A=\min\{A_1,A_2\}$. The function
$\rho(2-A)-2\rho\log\rho$ on $\rho\ge0$ has maximum
$2e^{-A/2}$, attained at $\rho=e^{-A/2}$.
This proves \eqref{eq:vertex-volume}.
\end{proof}

\begin{corollary}\label{cor:vertex-sum-rate}
Suppose the normalized Minkowski image of $A+A$ is contained in
$\mathcal B_\tau(dS)$ and distinct normalized algebraic integers are
separated by the interiors of translates of $\mathcal Q$.
Then, with $\mu=r_1/d$,
\[
 \log|A+A|\le d\left(\frac{\log S}{\tau}
             +\mu D_{1,\tau}+(1-\mu)D_{2,\tau}
             +E_\tau(S)\right)+O_{\tau,S}(\log(d+2)),
\]
where
\[
 D_{\kappa,\tau}
 =\frac1\kappa\log
       \frac{\Gamma(1+\kappa/\tau)}{\Gamma(1+\kappa)}
  +\frac1\tau\log\frac{e\tau}{\kappa}
  -\log\frac e{2\kappa}.
\]
\end{corollary}
\begin{proof}
The translates of the interior $\mathring{\mathcal Q}$ centered at the
elements of $A+A$ are disjoint and lie in
$\mathcal B_\tau(dS)+\mathcal Q$. Since
$\vol(\mathring{\mathcal Q})=\vol(\mathcal Q)$, divide the containing
volume by $\vol\mathcal Q$, apply Lemma~\ref{lem:vertex-packing}, and apply
\eqref{eq:projected-volume} with $\tau$ and with $1$.
The displayed constants follow by subtraction and Stirling's formula.
\end{proof}

\section{Ideal-smooth elements and a coprime additive factor}
\label{sec:sunits}

Let $L$ be a number field with at least one real embedding. Put
$d=[L:\Q]$, $n=r_1+r_2$, and choose one embedding $\sigma_i$ from each
complex conjugate pair. As usual, $\kappa_i=1$ at a real place and
$\kappa_i=2$ at a complex place, so $\sum_i\kappa_i=d$. Let
\[
 \mathcal H=\{y\in\R^n:\textstyle\sum_i\kappa_i y_i=0\},
 \qquad \Lambda_U=\logmap(\OO_L^\times)\subset\mathcal H.
\]
Let $w_L$ be the number of roots of unity in $L$.
Measures on $\mathcal H$ are Euclidean. We use the compact convex body
\[
 C=\{y\in\mathcal H:\textstyle\sum_i\kappa_i e^{\tau y_i}\le dM\},
 \qquad 0<\tau<1.
\]
The parameters $\alpha,M,\tau$ satisfy $M>B_\alpha$ as in
Lemma~\ref{geom:entropy}. Thus, for every fixed $H_0<H_\alpha$,
\begin{equation}\label{smooth:unit-volume}
 \log\vol_{n-1}(C)\ge n\log(H_0/\tau)-o(d).
\end{equation}

\subsection{The centred S-unit lattice}
Let $S=\{\mathfrak p_1,\ldots,\mathfrak p_m\}$ be distinct finite prime
ideals of $L$, let $\ell_j=\log\NN\mathfrak p_j$, and put
\[
 U_S=\{a\in L^\times:v_{\mathfrak p}(a)=0
                        \text{ for every finite }\mathfrak p\notin S\}.
\]
We use the convention $v_{\mathfrak p}(\mathfrak p)=1$. For $a\in U_S$,
write $e(a)=(v_{\mathfrak p_j}(a))_{j=1}^m$ and define
\[
 c(e)=\sum_{j=1}^m e_j\ell_j,\qquad
 y_i(a)=\log|\sigma_i(a)|-\frac{c(e(a))}{d}.
\]
The product formula gives
$\sum_i\kappa_i\log|\sigma_i(a)|=c(e(a))$, so $y(a)\in\mathcal H$.
The map
\[
 a\longmapsto (y(a),e(a))
\]
is a group homomorphism from $U_S$ to the additive group
$\mathcal H\times\R^m$.

\begin{lemma}\label{smooth:mixed-lattice}
Its kernel is the group of roots of unity of $L$, and its image
$\Gamma_S$ is a full lattice of rank $n-1+m$ in
$\mathcal H\times\R^m$. The valuation image
$\Lambda_v=e(U_S)\subset\Z^m$ has index
\[
 J_S=[\Z^m:\Lambda_v]\le h_L.
\]
Above every $e\in\Lambda_v$, the first coordinates of $\Gamma_S$ form
a translate of $\Lambda_U$.
\end{lemma}
\begin{proof}
The homomorphism
$\Z^m\to\Cl(L)$ sends $e$ to the class of
$\prod_j\mathfrak p_j^{e_j}$. Its kernel is precisely $\Lambda_v$:
an ideal supported on $S$ is principal if and only if its exponent
vector is the valuation vector of an $S$-unit. Its image has order at
most $h_L$, proving the index assertion.

Choose a $\Z$-basis of $\Lambda_v$ and an $S$-unit lifting each basis
vector. Together with fundamental ordinary units, their images form
a basis of $\Gamma_S$: subtraction of the selected lifts reduces
any $S$-unit to an ordinary unit. The valuation coordinates of the
lifts are linearly independent in $\R^m$, while the logarithms of
fundamental ordinary units form a $\mathbb Z$-basis of the unit lattice
$\Lambda_U$ and hence an $\mathbb R$-basis of $\mathcal H$ by
Dirichlet's unit theorem. These images therefore generate a full lattice. The same argument
shows that a fibre over $e$ is a translate of $\Lambda_U$. Finally, an element in the kernel is an ordinary unit all of whose
conjugates have absolute value one; by Kronecker's theorem it is a
root of unity.
\end{proof}

\subsection{Integer exponent truncations and a common origin}
Partition $S$ into finitely many families $S_a$ of cardinalities $m_a$.
Fix numbers $r_a>0$ and norm bounds $W_a>1$ such that
$\NN\mathfrak p_j\le W_a$ for $j\in S_a$. Define
\[
 E_a=\lfloor r_a m_a\rfloor,\qquad
 V=\{z\in\R^m:z_j\ge0,\
                 \textstyle\sum_{j\in S_a}z_j\le E_a
                 \text{ for every }a\}.
\]
This is a compact convex body; zero-dimensional families are omitted.
Its integer-point count is exactly
\begin{equation}\label{smooth:integer-count}
 N_V=|V\cap\Z^m|=\prod_a\binom{m_a+E_a}{m_a}.
\end{equation}

\begin{lemma}\label{smooth:count-doubling}
There are $u\in\mathcal H$ and $r\in\Z^m$ such that the full inverse
image $H$ in $U_S$ of
\[
 \Gamma_S\cap((u,r)+(C\times V))
\]
satisfies
\begin{equation}\label{smooth:H-lower}
 |H|\ge
 \frac{w_L N_V\vol_{n-1}(C)}{J_S\covol(\Lambda_U)}
 \quad\text{and}\quad
 |HH|\le 2^{n-1+m}|H|.
\end{equation}
For every $h\in H$, the vector $z(h)=e(h)-r$ is nonnegative and
belongs to $V$.
\end{lemma}
\begin{proof}
Average over the $J_S$ cosets of $\Lambda_v$ in $\Z^m$.
Every $z\in V\cap\Z^m$ lies in exactly one coset, so there is
$r\in\Z^m$ for which
\[
 |(r+V)\cap\Lambda_v|\ge N_V/J_S.
\]
For each such valuation vector $e$, choose a logarithmic lift
$y_e\in\mathcal H$. If $D_U$ is a measurable fundamental domain for
$\Lambda_U$, translation invariance and lattice tiling give
\[
 \int_{D_U}
 \sum_{e\in(r+V)\cap\Lambda_v}
 |(y_e+\Lambda_U)\cap(u+C)|\,du
 = |(r+V)\cap\Lambda_v|\vol_{n-1}(C).
\]
Since $\vol(D_U)=\covol(\Lambda_U)$, averaging over $D_U$ shows
that one common $u$ works for the sum of all valuation fibres; the
fibres need not be aligned. Every lattice point has $w_L$ inverse
images, proving the first estimate for some translate.

Now maximize the lattice-point count over all $u\in\mathcal H$ and
$r\in\Z^m$. The maximum exists. Indeed, the number of lattice points
in any translate of a fixed compact set is uniformly bounded: a
packing by small disjoint neighbourhoods of lattice points proves
this directly. The possible counts are thus bounded nonnegative
integers, so their supremum is attained by one of the allowed
translates. Use this maximizing pair to define $H$.

This maximum also bounds the count for every origin
$(u,s)\in\mathcal H\times\R^m$. If $(y,e)$ is a lattice point in
$(u,s)+(C\times V)$, then $e\in\Z^m$ and $e_j\ge s_j$ imply
$e_j\ge\lceil s_j\rceil$. Replacing $s$ componentwise by $\lceil s\rceil$
decreases every nonnegative displacement, so preserves all the
simplex inequalities. It does not change $y$ or $u$.

Let $\mathcal T=\Gamma_S\cap((u,r)+(C\times V))$. Convexity gives
$\mathcal T+\mathcal T\subset
\Gamma_S\cap(2(u,r)+2(C\times V))$. Partition the latter intersection
into its $2^{n-1+m}$ cosets modulo $2\Gamma_S$. In each coset,
subtracting a representative and dividing by two gives lattice
points in a translate of $C\times V$, possibly with noninteger
valuation origin. Each count is at most $|\mathcal T|$ by the
preceding paragraph. Thus
$|\mathcal T+\mathcal T|\le2^{n-1+m}|\mathcal T|$.
Since $H$ includes the full torsion fibre,
$|H|=w_L|\mathcal T|$ and $|HH|=w_L|\mathcal T+\mathcal T|$:
for any sum of logarithmic points, a root of unity can be supplied
by either factor. This proves the second estimate.
\end{proof}

For reference, the ordinary logarithm lattice satisfies
\begin{equation}\label{smooth:ordinary-covol}
 \covol(\Lambda_U)=2^{-r_2}\sqrt{r_1+4r_2}\,R_L.
\end{equation}
Indeed, the weighted logarithm matrix of fundamental units has
each deleted-row minor of absolute value $R_L$. For the unweighted
matrix the corresponding minor is
$\kappa_iR_L/2^{r_2}$; the Gram determinant and Cauchy--Binet give
the displayed formula. In particular
$J_S\covol(\Lambda_U)\le
\sqrt{r_1+4r_2}\,h_LR_L/2^{r_2}$.
Only this one class-group index is paid, even though all valuation
families are truncated simultaneously.

\subsection{A rank-two residue sieve}

\begin{lemma}\label{smooth:rank-sieve}
Let $B\subset\OO_L$ be finite and nonempty, and let $X\in\Z$ be nonzero
modulo every prime ideal in $S$. There are $b_0\in B$ and
$B'\subset B$ such that
\[
 |B'|\ge2^{-m}|B|,\qquad
 b-b_0+X\notin\mathfrak p_j
       \quad(b\in B',\ 1\le j\le m).
\]
\end{lemma}
\begin{proof}
In the finite field $E_j=\OO_L/\mathfrak p_j$, translation by $X$
has cycles of length equal to the residue characteristic. Properly
colour each cycle with at most three colours, and identify the
colours with the three nonzero vectors in the auxiliary vector
space $\mathbb F_2^2$. Write $v_j(a)$ for the vector assigned to
$a\in E_j$. Adjacent vectors are distinct, so
\[
 \det(v_j(a-X),v_j(a))=1\quad\text{in }\mathbb F_2.
\]
For $a\in B$, let $\bar a_j$ denote the image of $a$ in
$E_j=\OO_L/\mathfrak p_j$. On the actual elements $a,b\in B$, form
the matrix
\[
 K(a,b)=\prod_{j=1}^m
 \det(v_j(\bar a_j-X),v_j(\bar b_j)).
\]
Each determinant is bilinear in a two-dimensional row feature and
column feature. Their tensor products factor the matrix through
a vector space of dimension $2^m$. Hence its rank over $\mathbb F_2$
is at most $2^m$. Its diagonal entries are one. Moreover,
$K(a,b)=0$ whenever $\bar b_j=\bar a_j-X$ for some $j$.

Choose a basis of actual rows of $K$. Its row supports cover all
columns: otherwise a column vanishing on the basis rows would
vanish on every row, contrary to its diagonal entry. Therefore
some basis row has at least $|B|/2^m$ nonzero entries. Let $b_0$
be the element indexing this row and let $B'$ be its nonzero
column indices. The forbidden residue condition proves the lemma.
\end{proof}

Fix $0<\varepsilon<1/\sqrt5$ and a positive integer $X$ coprime to
all residue characteristics occurring in $S$. In the Minkowski
space, let $D_X$ be the product of real intervals of radius
$\varepsilon X/2$ and complex discs of the same radius. Averaging
over a fundamental domain of the integer lattice produces a
translate containing a set $B\subset\OO_L$ with
\[
 |B|\ge
 (\varepsilon X)^d(\pi/2)^{r_2}|\Delta_L|^{-1/2}.
\]
Apply Lemma~\ref{smooth:rank-sieve} and put
$P=B'-b_0+X$. All $p\in P$ satisfy
\begin{equation}\label{smooth:P-properties}
 |\sigma_i(p)-X|\le\varepsilon X,\qquad
 v_{\mathfrak p_j}(p)=0.
\end{equation}
The first assertion uses the half-radius of $D_X$ and the fact that
both $b$ and $b_0$ are in the same translate. If $\rd(L)\le\delta$,
we may take a subset of $P$ of precisely
\begin{equation}\label{smooth:P-count}
 \left\lfloor2^{-m}(\varepsilon X)^d
                  (\pi/2)^{r_2}\delta^{-d/2}\right\rfloor
\end{equation}
elements whenever this number is positive. This exact choice
will provide both a lower and an upper asymptotic estimate for $|P|$.

\begin{lemma}\label{smooth:injectivity}
For every finite $H\subset U_S$ and every $P$ satisfying
\eqref{smooth:P-properties}, multiplication $H\times P\to HP$
is injective.
\end{lemma}
\begin{proof}
Suppose $h_1p_1=h_2p_2$. The ratio
$w=h_2/h_1=p_1/p_2$ is an $S$-unit. Its valuations at $S$ are zero
by \eqref{smooth:P-properties}, so it is an ordinary unit.
For any embedding put $z=\sigma(w)$ and
$s_j=\sigma(p_j)/X$. Then $s_1=zs_2$ and
$|s_j-1|\le\varepsilon$, giving
\[
 |z-1|\le\varepsilon(1+|z|),\qquad
 \frac{1-\varepsilon}{1+\varepsilon}\le|z|
       \le\frac{1+\varepsilon}{1-\varepsilon}.
\]
Every conjugate of the algebraic integer $w+w^{-1}-2$ consequently
has modulus at most
\[
 \frac{|z-1|^2}{|z|}
 \le\varepsilon^2(|z|+2+|z|^{-1})
 \le\frac{4\varepsilon^2}{1-\varepsilon^2}<1.
\]
The norm of a nonzero algebraic integer is a nonzero rational integer.
Thus $w+w^{-1}-2=0$, so $(w-1)^2=0$ and $w=1$. Both factors are equal.
\end{proof}

\subsection{The physical norm and the sumset}
For the maximizing origins $u,r$ of Lemma~\ref{smooth:count-doubling},
put
\[
 I=\prod_{j=1}^m\mathfrak p_j^{r_j},\qquad
 c(r)=\log\NN I.
\]
Here $I$ may be fractional and $\NN I$ is a positive rational number.
Every element of $H$ belongs to $I$ because its valuations at $S$
are at least $r_j$ and its valuations elsewhere are zero.
Consequently $A=HP\subset I$, and $A+A\subset I$.

Scale Minkowski space by the diagonal real-linear map
\begin{equation}\label{smooth:physical-scaling}
 (z_i)\longmapsto
        (e^{-u_i-c(r)/d}z_i).
\end{equation}
Since $u\in\mathcal H$, its real determinant is
$e^{-c(r)}=(\NN I)^{-1}$; it need not be one.
For every nonzero $a\in I$, the integral ideal $(a)I^{-1}$ has norm
$|\NN_{L/\Q}a|/\NN I\ge1$. Therefore its scaled embedding vector
$z$ satisfies
\begin{equation}\label{smooth:physical-separation}
 \prod_i|z_i|^{\kappa_i}\ge1.
\end{equation}
This is the norm separation needed for packing.

For $h\in H$ set $z(h)=e(h)-r$ and
$c_h=\sum_jz_j(h)\ell_j$. The simplex truncations give
\[
 0\le c_h\le dT_d,\qquad
 T_d=\frac1d\sum_a E_a\log W_a.
\]
After scaling by \eqref{smooth:physical-scaling}, the absolute values
of $h$ are
$e^{c_h/d}e^{y_i(h)-u_i}$. The \emph{individually centred}
vectors $e^{y_i(h)-u_i}$ have norm product one and satisfy
\[
 \sum_i\kappa_i e^{\tau(y_i(h)-u_i)}\le dM.
\]
For $h_1,h_2\in H$, apply Lemma~\ref{geom:moment} with
\[
 a_i=e^{\tau(y_i(h_1)-u_i)},\qquad
 b_i=e^{\tau(y_i(h_2)-u_i)}.
\]
Then $\prod_i a_i^{\kappa_i}=\prod_i b_i^{\kappa_i}=1$ and
$\sum_i\kappa_i a_i,\sum_i\kappa_i b_i\le dM$, so
\[
 \sum_i\kappa_i
 \left(e^{y_i(h_1)-u_i}+e^{y_i(h_2)-u_i}\right)^\tau
 \le dF_\tau(M).
\]
Using also $e^{c_h/d}\le e^{T_d}$ and
$|\sigma_i(p)|\le(1+\varepsilon)X$ gives the claimed containment.
Every scaled element of $A+A$ consequently lies in
\begin{equation}\label{smooth:sum-body}
 \mathcal B_\tau(dS_d),\qquad
 S_d=F_\tau(M)(1+\varepsilon)^\tau X^\tau e^{\tau T_d},
\end{equation}
where
$\mathcal B_\tau(Y)=\{z:\sum_i\kappa_i|z_i|^\tau\le Y\}$.
It is essential to centre each vector before applying the
norm-one two-point inequality; the finite norm budget is restored
only through the final factor $e^{\tau T_d}$.

The open body
\[
 \mathcal Q=\{z:\textstyle\sum_i\kappa_i|z_i|<d/2\}
\]
has disjoint translates by distinct scaled points of $I$.
Indeed, an intersection would give a nonzero difference in the
scaled lattice with weighted sum of absolute values less than $d$.
Weighted arithmetic--geometric mean contradicts
\eqref{smooth:physical-separation}. Therefore
\begin{equation}\label{smooth:packing}
 |A+A|\le
 \frac{\vol(\mathcal B_\tau(dS_d)+\overline{\mathcal Q})}
      {\vol(\mathcal Q)}.
\end{equation}
All bodies are measurable and bounded; using the closure in the
numerator does not affect this disjoint-open-translates argument.

\subsection{Asymptotic rates}
Suppose $d\to\infty$, $r_1/d\to\mu$, $m_a/d\to\beta_a$, and assume
the arithmetic estimate
\[
 \limsup\frac1d\log(h_LR_L/2^{r_2})\le\mathscr R(\mu).
\]
The number of families and all their parameters are fixed along
this sequence. Write
\[
\begin{gathered}
 h(r)=(1+r)\log(1+r)-r\log r,\\
 s=\sum_a\beta_a h(r_a),\qquad
 T=\sum_a\beta_a r_a\log W_a,\qquad
 \gamma=\Bigl(\sum_a\beta_a\Bigr)\log2.
\end{gathered}
\]
Stirling's formula in \eqref{smooth:integer-count}, together with
\eqref{smooth:unit-volume} and \eqref{smooth:ordinary-covol}, gives
\begin{align}
 \log|H|&\ge(g(\mu)-o(1))d,&
 g(\mu)&=\frac{1+\mu}{2}\log(H_0/\tau)-\mathscr R(\mu)+s,
 \label{smooth:g-rate}\\
 |HH|&\le e^{(\lambda(\mu)+o(1))d}|H|,&
 \lambda(\mu)&=\frac{1+\mu}{2}\log2+\gamma.
 \label{smooth:lambda-rate}
\end{align}
Choosing $P$ by \eqref{smooth:P-count} gives, if $p(\mu)>0$,
\begin{equation}\label{smooth:p-rate}
 \log|P|=(p(\mu)+o(1))d,\qquad
 p(\mu)=\log X+\log\varepsilon-\tfrac12\log\delta
       +\tfrac{1-\mu}{2}\log(\pi/2)-\gamma.
\end{equation}
The outer-parallel-body bound of
Lemma~\ref{geom:outer-parallel}, applied to \eqref{smooth:packing},
gives
\begin{equation}\label{smooth:b-rate}
 \log|A+A|\le(b(\mu)+o(1))d,
\end{equation}
where, with $S=F_\tau(M)(1+\varepsilon)^\tau X^\tau e^{\tau T}$,
\[
 b(\mu)=\log X+T+\log(1+\varepsilon)
 +\frac{\log F_\tau(M)}{\tau}
 +\mu D_{1,\tau}+(1-\mu)D_{2,\tau}+E_\tau(S).
\]
Continuity of the volume bound in $S$ justifies replacing $S_d$ by
its limit: first use the fixed containing body with parameter $S+\eta$,
and then let $\eta$ decrease to zero.
Here $D_{\kappa,\tau}$ is the exact volume-ratio constant
\[
 D_{\kappa,\tau}
 =\frac1\kappa\log\frac{\Gamma(1+\kappa/\tau)}{\Gamma(1+\kappa)}
  +\frac1\tau\log\frac{e\tau}{\kappa}
  -\log\frac e{2\kappa}.
\]
This follows by polar integration:
\[
 \vol(\mathcal B_\tau(Y))
 =\frac{Y^{d/\tau}}{\Gamma(1+d/\tau)}
  \prod_i
  \frac{v_{\kappa_i}\Gamma(1+\kappa_i/\tau)}
       {\kappa_i^{\kappa_i/\tau}},
 \qquad v_1=2,\quad v_2=\pi,
\]
and Stirling's formula, with $\vol(\mathcal Q)$ obtained by
taking $\tau=1$ and $Y=d/2$.

\begin{proposition}\label{smooth:criterion}
Let $q>1$ and put $a(\mu)=g(\mu)+p(\mu)$.
If, for the limiting signature and the selected prime families,
\[
 p(\mu)>0,\qquad a(\mu)>0,\qquad
 b(\mu)<q\,a(\mu),\qquad
 p(\mu)+\lambda(\mu)<(q-1)a(\mu),
\]
then, for all sufficiently large degrees in this sequence,
\[
 \max\{|A+A|,|AA|\}\le |A|^q.
\]
The sets can be regarded as finite real sets, or as sets of algebraic
integers after a common rational-integer scaling.
\end{proposition}
\begin{proof}
Injectivity gives $|A|=|H||P|$, so
$\log|A|\ge(a(\mu)-o(1))d\to\infty$. Moreover
\[
 |AA|\le|HH||PP|
       \le e^{(\lambda(\mu)+o(1))d}|H||P|^2
       =|A|e^{(p(\mu)+\lambda(\mu)+o(1))d}.
\]
Together with \eqref{smooth:b-rate}, the strict inequalities absorb
the $o(d)$ errors and prove both bounds. An increase in $|H|$
only improves these comparisons; an upper estimate for $|H|$
is unnecessary.
Use the specified real embedding to obtain a real set. Since
$A\subset I$, choose a positive integer $D$ with $DI\subset\OO_L$.
Replacing $A$ by $DA$ makes all elements integral. It preserves
$|A|$ and $|A+A|$, while the product set becomes $D^2AA$ and therefore
has the same cardinality as $AA$.
\end{proof}

\section{The variational criterion and the exponent}
\label{sec:parameters}

Put $q_*=1.95835$ and $\sigma=2-q_*=0.04165$. The parameters are
\begin{equation}\label{eq:parameters}
\begin{gathered}
 c=0.0301,\quad \alpha=1.07187093,\quad \tau=0.09584694,\\
 H_0=4.63078205,\quad M=1.70613669,\quad
 \varepsilon=0.44721359,\\
 \delta=84.37558833,\quad C=16800001,\quad
 E_0=0.0000562,\quad \eta_0=0.0002.
\end{gathered}
\end{equation}
Write $U(c)$ for the constant $U$ in Proposition~\ref{arith:hR};
our $c$ equals $c_*$ of that proposition.
Here $5\varepsilon^2<1$, $H_0<H_\alpha$, and $M>B_\alpha$.
The latter comparisons are established by the rational interval
bounds in Appendix~\ref{app:intervals}.

Use the prime families from the arithmetic construction. For a
family with base norm $Q$, multiplicity $a\in\{1,2\}$, and residual
degree $f$, let
\[
 W=Q^f,\qquad \beta^{\rm full}=\frac{a}{10f},\qquad
 n_W=(W^\sigma-1)^{-1},\qquad
 Z(W)=-\log(1-W^{-\sigma})-\log2.
\]
Select $\lfloor\beta^{\rm full}d\rfloor$ descendant prime ideals if
$Z(W)>0$, and select none otherwise. Their norms are at most $W$.
Thus the selected density $\beta$ is $\beta^{\rm full}$ or zero.
This rule depends on the family and its residual degree, which are
fixed on the chosen cofinal subsequence. In a selected family, use
the truncation parameter $r_a=n_W$ in Section~\ref{sec:sunits}.

Let $\mathcal C_{Q,f}$ denote the signature correction for this
family in the arithmetic estimate (zero for a split pair), and put
\[
 v_{Q,f}=\frac{a}{10f}V(Q^f),\qquad
 K=\sum_{\text{families}}\mathcal C_{Q,f}.
\]
The summation includes all controlled families, including those
not selected for the smooth factor. In the notation of
Section~\ref{sec:sunits}, define
\begin{align*}
 G_\mu&=\frac{1+\mu}{2}\log(H_0/\tau)
             -\frac\mu2\log\pi-U(c)+\frac{c\mu}{2},\\
 C_\mu&=\log\varepsilon-\frac12\log\delta
                           +\frac{1-\mu}{2}\log(\pi/2),\\
 \lambda_0&=\frac{1+\mu}{2}\log2,\qquad
 \gamma=\sum\beta\log2,\\
 T&=\sum\beta n_W\log W,\qquad
 g=G_\mu-\sum v_{Q,f}-\mu K+\sum\beta h(n_W),\\
 \lambda&=\lambda_0+\gamma,\qquad
 B_0(\mu)=\frac{\log F_\tau(M)}\tau+\log(1+\varepsilon)
                   +\mu D_{1,\tau}+(1-\mu)D_{2,\tau},\\
 B(\mu)&=B_0(\mu)+E_0.
\end{align*}
The arithmetic estimate for $h_LR_L/2^{r_2}$ supplies this lower
rate $g$. Replacing $U(c)$ by a rigorous rational upper endpoint only
decreases $g$ and is permitted in all comparisons below.

\subsection{Balancing the two cardinality bounds}
Initially allow a real logarithmic centre $t$, and set
\begin{equation}\label{eq:balanced-centre}
 t=T+B(\mu)-g-2(C_\mu-\gamma)-\lambda,
 \qquad p=t+C_\mu-\gamma,\qquad a_*=g+p.
\end{equation}
The rates in Section~\ref{sec:sunits} then give
$b\le a_*+p+\lambda$, provided the packing error is at most $E_0$.
Consequently the two strict cardinality comparisons both follow
from
\begin{equation}\label{eq:Phi}
 \Phi=g-\lambda-\sigma a_*>0.
\end{equation}
The geometric-series entropy identity
\[
 h(n_W)-\sigma n_W\log W=-\log(1-W^{-\sigma})
\]
gives the useful separated form
\begin{equation}\label{eq:separated-criterion}
 \Phi=G_\mu-(1-\sigma)\lambda_0
       -\sigma(B(\mu)-C_\mu)
       +\sum\left\{\frac{a}{10f}Z(Q^f)_+
                        -v_{Q,f}-\mu\mathcal C_{Q,f}\right\}.
\end{equation}
In particular, the selection rule maximizes the contribution
obtained from the available family density in this bound.

\begin{lemma}\label{lem:nonnegative-new}
For every integer $u\ge2$,
\[
 Z(u)_+\ge V(u).
\]
\end{lemma}
\begin{proof}
$Z$ decreases with $u$, and $V(u)\le\log(u/(u-1))$.
Rational interval comparisons give
\[
 Z(7)>\log2,\qquad
 Z(10000)>\log(7/6),\qquad
 Z(C)>\log(10000/9999).
\]
These prove the assertion respectively on $[2,7]$, $(7,10000]$,
and $(10000,C)$. For $u\ge C$ the signed weight is nonpositive,
by the cutoff estimate in the arithmetic section. The displayed finite comparisons follow from the rational interval
evaluations described in Appendix~
\ref{app:intervals}.
\end{proof}

It follows that every newly controlled family has nonnegative
contribution to \eqref{eq:separated-criterion} at $\mu=0$.
For positive $\mu$, only its signature correction must be paid.
The nine ramified primes and seven original split pairs are kept
individually, with the residual-degree branches in their original
cuts. Their contributions are bounded below by the minimum over
those branches. For all other families, we use
Lemma~\ref{lem:nonnegative-new} and the total inert correction
bound from the arithmetic section. Since a minimum of finitely
many affine functions is concave, the resulting lower bound for
$\Phi$ is concave on $0\le\mu\le3/5$. Its two endpoint values
therefore suffice.

\subsection{Positive rates and the packing bootstrap}
We check the centre and the sizes separately, so that the
packing estimate is not used circularly. For a selected family put
$L=\log W$. Its contributions to $t$ and $p$ in
\eqref{eq:balanced-centre}, apart from the signature correction,
are
\[
 v_{Q,f}+\beta\{n_WL-h(n_W)+\log2\},\qquad
 v_{Q,f}+\beta\{n_WL-h(n_W)\},
\]
respectively. The signature correction adds
$\mu\mathcal C_{Q,f}$ to both.
For a selected family, $\sigma L<\log2$ and $L\ge\log2$.
Thus, with $x=\sigma L$,
\[
 \sigma\log2\le x<\log2.
\]
The function $x/(e^x-1)$ is decreasing, whereas
$\log(1-e^{-x})$ is increasing. Hence
\begin{align*}
 n_WL-h(n_W)
 &=\frac{1-\sigma}{\sigma}\frac{x}{e^x-1}
                         +\log(1-e^{-x})\\
 &\ge\frac{1-\sigma}{\sigma}\log2
                         +\log(1-2^{-\sigma})>1.
\end{align*}
Also
\[
 V(W)\ge-\frac{2c\log W}{W-1}
                       \ge-2c\log2>-1.
\]
Thus these contributions are nonnegative. Moreover, for a selected
family,
\[
 h(n_W)=\sigma n_W\log W-\log(1-W^{-\sigma})
       > Z(W)=Z(W)_+\ge V(W),
\]
where the last inequality is Lemma~\ref{lem:nonnegative-new}. Hence
$\beta h(n_W)-v_{Q,f}=\beta(h(n_W)-V(W))\ge0$.

For an inactive family, $W\ge u_*=2^{1/\sigma}$. The coarse
prime count from Section~\ref{sec:arithmetic} gives at most $10C$
controlled primes of $F$; equivalently, the family multiplicities
satisfy $\sum a\le10C$. The total possible loss to $t$ or $p$ from
inactive negative Euler terms is therefore at most, using the rigorous
integer lower bound $u_0=16889272<u_*$,
\begin{equation}\label{eq:inactive-loss}
 D_* = C\frac{2c\log u_0}{u_0-1}
       <0.996564355.
\end{equation}
Here $\log u/(u-1)$ decreases for $u>1$, and each family coefficient
is $a/(10f)\le a/10$. Keeping the original families' minimum
contributions before applying this coarse loss gives the rate
bounds in Table~\ref{tab:final}. For $g$, it suffices to subtract
the total inert and ramified corrections from $G_\mu$.
All these lower bounds are affine or concave in $\mu$, so their
endpoint checks cover the interval. To bound the balanced centre
from below we drop the positive packing term $E_0$; in particular
this step is independent of the packing estimate.

We obtain $t>17$, $T\ge0$, and hence
\[
 S=F_\tau(M)(1+\varepsilon)^\tau e^{\tau(t+T)}
    \ge F_\tau(M)(1+\varepsilon)^\tau e^{17\tau}.
\]
The explicit outer-parallel error decreases with $S$. Rational
intervals in its formula now give
\begin{equation}\label{eq:packing-numeric}
 E_\tau(S)<0.00005611486<E_0.
\end{equation}
This proves the required packing bound.

\subsection{An integral centre and completion of the proof}
The additive sieve requires an integer $X$ coprime to all selected
residue characteristics. Since $e^{17}>C$ and $e^{17}>396738$,
Dusart's unconditional prime-interval estimate
\cite[Proposition~6.8]{Dus} supplies a rational prime $X$ with
\[
 e^t<X\le e^t\left(1+\frac1{25t^2}\right).
\]
Thus $X>C$ and
\[
 0<\log X-t<\frac1{25\cdot17^2}<\eta_0.
\]
All selected residue characteristics are smaller than $C$, so $X$
is coprime to them. Increasing the logarithmic centre by $\eta$
increases $p$, $a_*$, and the sumset bound by $\eta$, while the
packing error only decreases. The product comparison loses
$\sigma\eta$ and the sumset comparison improves by
$(q_*-1)\eta$. It is therefore enough to subtract
$\sigma\eta_0$ from the rigorous lower bound for $\Phi$.

\begin{table}[ht]
\centering
\caption{Strict rational lower bounds. The criterion includes
the full centre-rounding loss $\sigma\eta_0$. The centre and size
columns were bounded with the packing term omitted.}
\label{tab:final}
\begin{tabular}{ccccc}
\toprule
$\mu$ & $\Phi-\sigma\eta_0$ & $t$ & $p$ & $g$\\
\midrule
$0$ & $0.0001795984$ & $18.0810904719$ & $15.2325377565$ & $0.9667653864$\\
$3/5$ & $0.0960506190$ & $17.6455208284$ & $14.6614933014$ & $0.9714028747$\\
\bottomrule
\end{tabular}
\end{table}

The outward rational interval computations prove the strict
inequalities in Table~\ref{tab:final}; Appendix~\ref{app:intervals}
specifies the required estimates and rounding conventions. Concavity gives strict positive margins throughout
$[0,3/5]$. Choose the prime centre just described and apply
Proposition~\ref{smooth:criterion}. The margins absorb all
asymptotic $o(d)$ errors. The resulting sets have cardinalities
tending to infinity. A common positive integer scaling clears
the fractional ideal denominators and preserves the cardinalities
of the set, its sumset, and its product set. This proves
Theorem~\ref{thm:main}.

\begin{remark}
The proof does not assert optimality of this exponent. In
particular, the maximal-translate doubling bound for the mixed
$S$-unit lattice costs $\log2$ per valuation dimension. The
integer-point entropy of a simplex alone does not justify a
smaller doubling constant for an arbitrary class-group valuation
sublattice. Removing the additive factor would also require a
different global sumset argument.
\end{remark}

\appendix
\section{Rational interval computations}\label{app:intervals}

All numerical assertions in this appendix are evaluated using exact
integers and rational intervals. A terminating decimal endpoint is
interpreted as an exact rational number. For an interval $[l,u]$, the
lower endpoint is rounded downward and the upper endpoint upward to
multiples of $10^{-32}$, using integer division. Sums, products,
quotients with nonzero denominators, and monotone functions are evaluated
with outward rounding. Thus every strict numerical comparison used below
is reduced to a comparison of rational endpoints.

For $1\le x\le2$, put $z=(x-1)/(x+1)$. We use
\begin{equation}\label{cert:log}
 \log x=2\sum_{j=0}^{N-1}\frac{z^{2j+1}}{2j+1}+R_N,\qquad
 0\le R_N\le\frac{2z^{2N+1}}{(2N+1)(1-z^2)}.
\end{equation}
A positive rational argument is reduced to $[1,2]$ by an integer
power of $2$. The main comparisons use $N=45$.
Machin's identity
\[
 \pi=16\arctan(1/5)-4\arctan(1/239)
\]
and alternating Taylor series enclose $\pi$.
For Euler's constant, let $H_N=\sum_{j=1}^N1/j$ and put
\[
 J_N=H_N-\log N-\frac1{2N}+\frac1{12N^2}-\frac1{120N^4}.
\]
Using $\psi(N+1)=H_N-\gamma_E$, the signed digamma remainder
\cite[\S5.11(ii)]{DLMF} gives
\[
 J_N<\gamma_E<J_N+\frac1{252N^6}.
\]
We take $N=1000$.

For an integer $k\ge3$, write $Z_{k,N}=\sum_{j=1}^Nj^{-k}$.
The integral test gives
\[
 Z_{k,N}+\frac1{(k-1)(N+1)^{k-1}}
 <\zeta(k)<
 Z_{k,N}+\frac1{(k-1)N^{k-1}}.
\]
We take $N=2000$ and use $\zeta(2)=\pi^2/6$.
For $0<z<1$, the expansions
\cite[eqs.~(5.7.3)--(5.7.4)]{DLMF} and the Gamma functional equation give
\begin{align*}
 \log\Gamma(z)&=-\log z-\gamma_Ez+
      \sum_{k=2}^\infty\frac{(-1)^k\zeta(k)}kz^k,\\
 \psi(z)&=-\frac1z-\gamma_E+
      \sum_{k=2}^\infty(-1)^k\zeta(k)z^{k-1}.
\end{align*}
After term $k=100$, the absolute remainders are at most
$2z^{101}/(101(1-z))$ and $2z^{100}/(1-z)$, since $\zeta(k)<2$.
For arguments close to $1$, one can instead use the corresponding
series for $\log\Gamma(1+z)$ with $|z|\le1/2$ and the same absolute
tail estimate. Positive integer arguments are evaluated by factorials;
other arguments are reduced by $\Gamma(x+1)=x\Gamma(x)$.

For the entropy constants, combining the expansions gives, for $0<z<1$,
\begin{align*}
 \log H_z&=z+1-\log z+
   \sum_{k=2}^\infty(-1)^{k+1}\frac{k-1}{k}\zeta(k)z^k,\\
 \log B_z&=\log z+\frac1z+\gamma_E+
   \sum_{k=2}^\infty(-1)^{k+1}\zeta(k)z^{k-1}.
\end{align*}
Their absolute tail bounds after term $k=100$ are
$2z^{101}/(1-z)$ and $2z^{100}/(1-z)$.
The parameter $\alpha=1.07187093$ is greater than $1$, so we apply
these series at $z=\alpha-1$ and use the exact recurrences
\[
 \log H_\alpha=\log H_z+\log B_z-\frac1z,\qquad
 \log B_\alpha=\log B_z+\log(\alpha/z)-\frac1z.
\]
These follow immediately from
$H_x=\Gamma(x)e^{x(1-\psi(x))}$,
$B_x=xe^{-\psi(x)}$, and
$\psi(x+1)=\psi(x)+1/x$.
The resulting intervals verify
\[
 4.63078206355<H_\alpha<4.63078206356,\qquad
 1.70613667289<B_\alpha<1.70613667290.
\]

Square roots are enclosed by integer square roots:
for a nonnegative rational $x$, if
$m=\lfloor\sqrt{\lfloor10^{64}x\rfloor}\rfloor$, then
\[
 m/10^{32}\le\sqrt{x}<(m+1)/10^{32}.
\]
For $0\le x<242$, with
$s=\sum_{j=0}^{240}x^j/j!$ and $u=x^{240}/240!$, we have
\[
 s\le e^x\le s+\frac{ux}{241(1-x/242)}.
\]
Negative exponential arguments are handled using $e^{-x}=1/e^x$.
For stable evaluation of the two-point moment, put
$\ell=\log(M+\sqrt{M^2-1})$ and use
\[
 \log F_\tau(M)=\ell+
           \tau\log(1+e^{-2\ell/\tau}).
\]

For the Euler weights at small integer arguments, a finite sum
through an integer cutoff $N_0$ is completed with
\[
 0<\sum_{m=N_0+1}^{\infty}\frac1{u^m+1}
      <\frac1{u^{N_0}(u-1)}.
\]
For $u\ge1000$ the wider, faster intervals
\[
 \frac1u\le b(u)\le\frac1{u-1},\qquad
 \frac{2\log u}{u+1}\le a(u)\le\frac{2\log u}{u-1}
\]
are sufficient. All branch minima are bounded below by the minimum
of their rational lower endpoints; correction maxima are bounded
above by the maximum of their upper endpoints.

The finite-field data used above are obtained directly from the
defining polynomial, rather than from a precomputed numerical table.
The prime sieve and the integer-sum tail bounds in the arithmetic section
are evaluated from the explicit formulas given there. Together with the
outward-rounding rules above, these exact computations establish the GS
inequality, the cutoff, the selection inequality, the entropy-body
constants, both signature endpoints, positivity of the rates, the packing
error, and the centre-rounding loss. They justify every strict decimal
bound stated in the manuscript.

\section*{Use of AI-assisted tools}

Generative AI tools were used to assist with selected computations,
literature searches, and \LaTeX{} editing. The mathematical ideas,
constructions, proofs, and conclusions were developed by the authors.
References identified with AI assistance were checked against the original
sources, and any AI-assisted computational output used in the manuscript
was independently verified. The numerical inequalities used in the proof
are established by the exact rational interval estimates described in
Appendix~A. The authors are responsible for the content of the manuscript.

\end{document}